\documentclass[conference,twocolumn,10pt,final]{IEEEtran}

\usepackage{inputenc}
\usepackage{graphicx, subfigure}
\usepackage[cmex10]{amsmath}
\usepackage{color}
\usepackage{amssymb}
\usepackage{cite}
\usepackage{mdwmath}
\usepackage{multirow}
\usepackage{mdwtab}
\usepackage{verbatim}

\newcommand{\FFF}{{\mathcal{F} }} 
\newcommand{\RRR}{{\mathrm I\! \textsc{R} }} 
 
\newcommand{\XXX}{{\mathcal{X} }} 
\newcommand{\TTT}{{\mathcal{T} }} 
\newcommand{\YYY}{{\mathcal{Y} }} 
\newcommand{\ZZZ}{{\mathcal{Z} }} 
\newcommand{\QQQ}{{\mathcal{Q} }}
\newcommand{\WWW}{{\mathcal{W} }} 
 
\newcommand{\UUU}{{\mathcal{U} }} 
\newcommand{\AAA}{{\mathcal{A} }}

\newcommand{\ZZ}{{\mathcal{Z}^\# }} 

\newcommand{\EEE}{{\mathrm I\! \textsc{E} }} 
\newcommand{\PPP}{{\mathrm I\! \textsc{P} }} 
 
\newcommand{\PP}{{\mathbf P }}

\newcommand{\probaSpace}{{(\Omega,\FFF,\PPP) }} 
\newcommand{\omegaNav}{{\Omega^\text{NAV} }}

\newcommand{\de}{{\delta }}

\newcommand{\al}{{\alpha }}

\newtheorem{theorem}{\textsc{ \textbf{Theorem} }}[section]
\newtheorem{definition}{\textsc{ \textbf{Definition} }}[section]

\newtheorem{remarque}{\textsc{ \textbf{Remark} }}[section]
\newtheorem{proposition}{\textsc{ \textbf{Proposition} }}[section]

\newtheorem{demonstration}{\textsc{ \textbf{Proof} }}
\newtheorem{algorithm}{\textsc{ \textbf{Algorithm} }}[section]

\newenvironment{ALGORITHM}{ \begin{algorithm} }{ \end{algorithm} }

\title{
A Stochastic Dynamic Principle for Hybrid Systems with Execution Delay and Decision Lags 
 }

\author{\IEEEauthorblockN{K. Aouchiche\IEEEauthorrefmark{1}, F. Bonnans\IEEEauthorrefmark{2}, G. Granato\IEEEauthorrefmark{1}, H. Zidani\IEEEauthorrefmark{3}}
\IEEEauthorblockA{\IEEEauthorrefmark{1}Renault SAS. 
kamal.aouchiche@renault.com;
giovanni.granato@renault.com}
\IEEEauthorblockA{\IEEEauthorrefmark{2}
INRIA Saclay \& CMAP.
frederic.bonnans@inria.fr}
\IEEEauthorblockA{\IEEEauthorrefmark{3}
ENSTA ParisTech \& INRIA Saclay.
hasnaa.zidani@ensta-paristech.fr}}

\begin{document}

	\maketitle		

	\begin{abstract}
	This work presents a stochastic dynamic programming (SDP) algorithm that aims at minimizing an economic criteria based on the total energy consumption of a range extender electric vehicle (REEV).
	This algorithm integrates information from the REEV's navigation system in order to obtain some information about future expected vehicle speed. 
	The model of the vehicle's energetic system, which consists of a high-voltage (HV) battery, the main energy source, and an internal combustion engine (ICE), working as an auxiliary energy source), is written as a hybrid dynamical system and the associated optimization problem in the hybrid optimal control framework.
	The hybrid optimal control problem includes two important physical constraints on the ICE, namely, an activation delay and a decision lag. 
	Three methods for the inclusion of such physical constraints are studied. 
	After introducing the SDP algorithm formulation we comment on numerical results of the stochastic algorithm and its deterministic counterpart.
	\end{abstract}
	
	\section{Introduction}
	
	Electrified automotive powertrain technology is being greatly developed as an increasingly number of carmakers wish to adopt vehicles with an electrification of the powertrain as a viable solution for reducing greenhouse gas emissions worldwide to meet stringer regulative legislation and consumers' demand. 
	There is a strong effort of major constructors in order to deploy fully electric vehicles (EVs) as early as possible in some car market segments. 
	This work focus on a specific class of EVs, namely, range extender electric vehicles (REEV). 
	This study aims at synthesizing a supervising optimal control strategy of the range extender (RE) using information from the vehicle navigation system (NAV) as well as a statistical analysis of previously stored driving data.
	One important feature of this study is the incorporation of two RE's physical constraints -- the execution lag and the decision delay -- in the optimal control problem formulation. 
	From the mathematical point of view, the system is modeled as a hybrid dynamical system in which discrete decisions may suffer from an execution delay and a decision lag. 
	The next step is the formulation of a stochastic optimal control problem in this hybrid framework alongside a stochastic dynamic programming principle which is then used to obtain an optimal controller.%

	 In the literature, there are several works that deal with the synthesis of optimal -- or sub-optimal -- supervisory control strategies for a vehicle with a certain degree of hybridization.
	These works include heuristics such as model-based \cite{sciarretta} or fuzzy logic algorithms \cite{baumann}.
	Attention is also given to adaptive control techniques \cite{beck2007}.
	Works \cite{chanchiao}, \cite{deantatejr} use stochastic dynamic programming algorithms to synthesize control strategies that are not cycle-specific but do not include any information from the navigation system. 
 	
	 This study seeks to synthesize an optimal control sequence of the ICE in order to minimize an economic cost criteria related to overall energy consumption. 
	In this case, a rather realistic scenario is that there is no precise information about how the vehicle will behave -- i.e. what its power demand will be -- on a particular trip, even if there is some information available from a navigation system. 
	
	Many modern control systems involve some high-level logical decision making process coupled with underlying low-level continuous processes \cite{Varaiya93smartcars}, \cite{Lygeros}, \cite{Dharmatti}, \cite{Back}. 
	Some of those are flight control systems, production systems, chemical processes and traffic management systems. 
	The term hybrid stems from the different nature of the systems' evolution, continuous and discrete. 
	Hybrid systems have some supervision logic that intervenes punctually between two or more continuous functions. 
	An important share  of this research seeks to extend the well known theory of continuous systems - for instance, the maximum principle \cite{Garavello}, \cite{Arutyunov} and the dynamic programming principle \cite{Dolcetta:143} - to hybrid systems. 
	Several authors \cite{Branicky}, \cite{Zhang}, \cite{bens} propose different modeling frameworks for hybrid dynamical systems varying in degree of generality. 
	This study bases itself on the quite general framework \cite{bens}. 
	Many of the applications appearing in the literature can be viewed as their particular cases.	
	Works \cite{Branicky}, \cite{Zhang}, \cite{bens}, \cite{Shaikh} make considerable effort towards modeling and simulating hybrid dynamical systems in both deterministic and stochastic cases. 
	Continuous hybrid systems including decision lags and executions delays are studied in \cite{pham}, \cite{pAsea_pZak1999} where a suitable treatment is given to describe these systems.
	
	The contribution of this article is the use of a stochastic speed model in an optimal hybrid control framework, incorporating NAV information, that integrates the RE's activation delay and decision lag constraints.
	
	\section{Power Management Strategy for Range Extender Electric Vehicles}

	\subsection{Range Extender Electric Vehicles}
	
	This section discusses the model of the power management system of a REEV that is to be optimized and introduces the notation used throughout the document.
	A REEV is a vehicle that combines a primary power source - a HV battery - and a small dimensioned (powerwise) secondary power source - in our case an ICE. 
	The traction (or propulsion) of the vehicle is performed by an electric motor connected to the vehicle's wheels. 
	Both power sources may supply the energy demanded by the driver. 
	Additionally, since the model considers a range-extender electric vehicle type, it cannot rely solely on the RE's power to drive the vehicle.
	The architecture is that of a series hybrid electric vehicle, which means that the ICE is not mechanically connected  to the transmission.
	Instead, a generator transforms the mechanical energy produced in the ICE into electric current that can be directed towards the electric motor or charge the battery. 
	Also, the powertrain components are represented by a quasi-static models, detailed in \cite{guzzella}, thus neglecting any transient response.
	Controls available include the turning on and off of the ICE and the power produced in it. 
	Due the discrete nature of the switch on/off control and the continuous nature off the power control, the control variable is seen as a hybrid control. 
	The controlled variables are the state of charge (SOC) of the battery and the ICE state on or off.
	A power management strategy (PMS) for a REEV is a control sequence that dictates the state of the ICE - on or off - and if it is on, how much power it will supply to the electric motor. 
	

	\subsection{Stochastic Model}
	The instant SOC evolution depends on the vehicle's instant power demand. 
	To devise efficient PMSs, the control synthesizer needs information about future power demands of the system as close as possible to the real power demand to be requested. 
	Hence, given a fixed route, since the control synthesizer does not rely on exact before-drive speed knowledge, the PMS must be robust enough to "absorb" most of the situations occurring in a real driving cycle.
	In this case presented here, it contains information about possible deviations from the speed suggested by the NAV. 
	The NAV outputs useful information is in the form of a geographic route. 
	After the driver has entered a geographic location as a destination point in the NAV, it suggests a route, calculated using some optimization algorithm. 
	A working hypothesis is that the vehicle will follow this route suggested by the NAV.
	The route is an assemblage of links, all of which have an associated constant recommended cruise speed.
	The route segmentation is used as discretization of the optimal control problem.
	The index $k=1,\cdots,K$ refers to route links nodes. 
	Each link has an entry point and an exit point, thus, for a route with $K$ links we have $K+1$ nodes. 
	Index $k$ indexes a variable at a link's entry node. 
	Then, at the $k$-th link's entry point, one has the battery SOC $x_k$ , the vehicle instant speed $y_k$ and the ICE state $q_k$ . 
	In the same manner, the ICE controls will be applied at each link entrance, denoted by $u_k$, the ICE power, and $w_k$, the ICE on/off switch decision.
	
	Each link's speed depends on particular characteristics of the route segment considered, such as type, number of lanes, location and may be dynamic, depending on weather conditions or the hour of the day. 
	The speed profile formed with the speed suggested for each route link the \emph{NAV speed profile}. 
	As a result of the disparity between the real driving speed - not known \textit{a priori }- and the NAV speed profile, the knowledge on how much power will be instantaneously required by the vehicle in a car trip is not known. 
	Let $\omegaNav$ be the -- finite and discrete -- set of possible NAV properties values, consisting for instance of number of lanes, location and hour of the day. 
	Then, for a property set $\eta \in \omegaNav$, let $\Omega(\eta)$ be the space of possible vehicle speeds.
	For example, if $\eta$ corresponds to a $3$ lanes link located on a freeway at $14$ p.m., one possible speed set may be $\Omega(\eta) = \{90,110,130,150\}$.
	For a fixed $\eta$, let $\FFF$ be a $\sigma$-algebra over $\Omega(\eta)$ and $\PPP$ a probability. 
	Then, we define $\xi : \Omega(\eta) \rightarrow \Omega(\eta)$ to be a random variable over the probability space $\probaSpace(\eta)$ -- we denote by $\xi$ the random variable and its particular realization, the utilization will be clear from the context. 
	Given two sets of properties $\eta, \eta' \in \omegaNav$ (representing two neighboring links) and two speed values $y \in \Omega(\eta)$, $y' \in \Omega(\eta')$, we wish to define the probability of changing to speed $y'$ from speed $y$.
	 It is natural to consider that the transition probability from $y$ to $y'$ depends on $\eta$ and $\eta'$. For all $\eta, \eta' \in \omegaNav$, we define a transition probability $\PP$ between $y$ and $y'$ to be $\PP(y,y') = \PPP(\xi' = y' ~|~ \xi=y, \eta, \eta' )$.
	Such properties should be tailored to contain relevant useful information. For instance, one may not need to consider time frames $15'$ apart, whereas a distinction between rush hours and lower traffic periods might be a more interesting choice.

	\subsection{Execution Delay and Decision Lag}
	
	This section explains the motivation of including the execution delay and the decision lag constraints in the model, as well as  the method used to incorporate these constraints in the optimal hybrid control problem.
	Frequent switching of the RE is undesirable in order to avoid mechanical wear off of the RE and acoustic nuisance for the driver. 
	Decision lags impose that some minimum time is to be respected before turning the RE on and thus, avoid this issues.
	Also, the RE ability of delivering power to the vehicle's HV network is constrained by the catalytic converter's temperature.
	Indeed, the catalytic converter present in the ICE's exhaust pipe must undergo a warm-up period in order to achieve a satisfactory working point. 
	While this is not the case, the RE must avoid functioning in any working point different than the minimal one -- required for not stalling -- and thus, cannot provide any auxiliary power to the system.
	The activation delay models that behavior.
	In this application, however distinct in nature, the decision lag and the execution delay have the same order of magnitude of about $120 s$. 
	Thus, they are conveniently regrouped in only one delay/lag constant variable $\delta$, expressed in time units.
	
	These constraints are introduced through a state variable at each link entry point, $t_k \in [0,\delta]$, representing a counting time since the last turn-RE-on decision. 
	In a nutshell, if $t_k < \delta$, no switches are available to the controller and no power can be issued from the RE. 
	Whenever $t_k \geq \delta$, switches are available and the RE can provide auxiliary power to the vehicle.

	\subsection{Hybrid State and Controls}

	A hybrid state is a state vector consisting of continuous as well as discrete valued variables. 
 	The continuous variables are the battery SOC $x_k \in [0,1]$, the vehicle instant speed $y_k \in [0,y_{\text{max}}]$ and the counting time $t_k \in [0,\delta]$, and the discrete variable is the ICE state $q_k \in \{0,1\}$. 
	Let $\ZZZ = \XXX \times \YYY \times \TTT \times \QQQ$ be the hybrid state space and $z_k = (x_k,y_k,t_k,q_k) \in \ZZZ$ be the hybrid state at each link's entry point.
	Let $w_k$ be the discrete control applied to the system at link $k$. 
	The discrete controls are used to switch between continuous modes of the hybrid system. 
	They are valued in a discrete set $\WWW(z_k) \subset \{0,1\}$ in the case of a two mode system. 
	The control $w_k=0$ means no particular order, simply leaving whatever mode is on active, while $w_k=1$ represents a mode switch. 
	We recall that only off$\rightarrow$on transitions are delayed in time. 
	Mode transitions are decided at link $k$ but are executed afterwards and between its decision and execution instants, no other switch order can be decided. 
 	Continuous controls $u_k$ are valued in a continuous set $\UUU(z_k)$. 
	Let $a_k = (u_k,w_k)$ be the hybrid control valued in $\AAA(z_k) = \UUU(z_k) \times \WWW(z_k)$, the hybrid control space. 
 	 \begin{definition}
		  The control sequence $(a_1, \cdots, a_K)$ is an \emph{admissible hybrid control sequence} if it satisfies the relations 
		  $(1)$ For $(z_1, \cdots, z_K) \in \ZZZ^K$, we have $(a_1, \cdots, a_K) \in \AAA(z_1) \times \cdot \times \AAA(z_K)$.
	  	$(2)$ Let $\tau = \{ k ~|~ w_k=1 \}$ be the set of links in which a switch decision is made. Then, $\forall i \in \tau$, $t_i \geq \de$.
		  $(3)$ $t_{K+1} > \de$.
		\label{admissibleControl} 
 	 \end{definition}
	The first condition states that every hybrid control should be in the hybrid control domain. 
	The second condition is the decision lag whereas the third one implies that no order that suffers from an execution delay can be decided and not executed.
 	For a hybrid state $z = (x,y,t,q) \in \ZZZ$, we can write the control space as follows: 
  \begin{eqnarray}
		\UUU(z)&=& \left\{ \begin{array}{ccc}
			0 & \text{if} & t < \delta \text{ or } q=0 \\
			\UUU(y) & \text{if} & t \geq \delta \text{ and } q=1 \\
		\end{array} \right.\\
  	W(z)&=& \left\{ \begin{array}{ccc}
			0 & \text{if} & t < \delta \\
			\{0,1\} & \text{if} & t \geq \delta \\
		\end{array} \right.
	\label{controlPendingOrder}
	\end{eqnarray}
  
  \subsection{Hybrid State Evolution and Control Policies}
    
  	This section describes the evolution equations of the state vector components from a known value $z_k$ at link $k$. 
	The speed on the next link is a random variable $\xi_k > 0$. 
	The SOC evolution depends on the power produced by the ICE $u_k$ and $\xi_k$, as well as on $x_k$ and $q_k$. 
	The time since last activation on the next link depends on $w_k$, $t_k$, $\xi_k$ and on the link length $d_k$, given by the NAV. 
	The evolution of the discrete variable $q_k$ depends on decisions made in the last link. 
	This is summarized as follows:
	\begin{eqnarray}
		X_{k+1} &=& f(z_k,u_k,\xi_k,q_k) \label{xeq}\\
		Y_{k+1} &=& \xi_k \\
		T_{k+1} &=& \left\{ \begin{array}{ccc}
			t_k + d_k/\xi_k & \text{if} & w_k=0 \\
			d_k/\xi_k & \text{if} & w_k=1 \\
		\end{array} \right. \label{teq}\\
  	q_{k+1} &=& g(q_k,w_k) \label{qeq}
	\label{stateEvolution}
	\end{eqnarray}
	Given a known state vector $z_k$, the next speed value is not exactly known, but it depends on the next link's set of properties $\eta_{k+1}$ given by the NAV.
	The transition probability $\PP$ dictates the random evolution of the system state vector, i.e., the random process $(Z_k)$, where the $Z_k$ are $\FFF_k$-measurable for all $k \geq 1$.
	An admissible control sequence depends on the state value at future times, which are unpredictable, and so are the future control domains. 
	Nonetheless, one can define of a sequence of hybrid control laws, or a \emph{hybrid policy} $(\al_0,\cdots,\al_{T-1})$ where each $\al_k$ is a function of the hybrid state into a hybrid control $a_k$:
	
	\begin{definition} 
		A hybrid policy $\pi = (\al_1,\cdots,\al_{K})$ is a sequence of functions $\al$ where
		\begin{eqnarray}
			 \al &:&  \ZZZ \rightarrow \AAA \\
			\nonumber        & & z \mapsto \al(z) = a.
		\end{eqnarray}			
  	A policy $\pi$ is said to be an \emph{admissible hybrid policy wrt $z$} if, given a initial state $z$, the sequence $\pi=(\al_1(z),\cdots,\al_K(z_K))$ is an admissible hybrid control sequence for all $z_2,\cdots,z_K \in \ZZZ$. 
	Let $\Pi(z)$ be the set of all admissible hybrid policies with respect to the initial state $z$.
  \end{definition}	
	As a consequence, if a policy $\pi \in \Pi(z)$ is admissible, the controls produced by it are an admissible hybrid control sequence.
		
	\subsection{Cost Functions}
	Given a hybrid state $z$ and hybrid control $a$, let $\l(z,a) \in \RRR$ be an instant cost function. 
	It is well advised to make explicit the separation of the instant cost function in its continuous and discrete control associated components. 
	Given a hybrid control $a=(u,w)$, denote $\l(z,a) = \sigma(u) + \rho(w)$, where $\sigma$ is the continuous control cost component and $\rho$ is the cost due mode switching.
	For $\beta >0$, define a final cost function to be evaluated at the final state $z=(x,y,t,k)$ to be $\phi(z) = -\beta x$. 
	The final cost function $\phi(\cdot)$ is assumed to have a finite expected value.
	The parameter $\beta$ works as a scaling factor adjusting the relative value of the electricity and fuel consumption and can be seen as reflecting the economic price of $1\%$ of SOC relative to $1l$ of fuel.
	\begin{definition}
		Given an initial state $z$ and an admissible policy $\pi \in \Pi(z)$, the states $Z_1,\cdots,Z_K$ are random variables given by 
		\begin{equation}
			Z_1 = z,~~ Z_{k+1} = h(Z_k,a_k,\xi_k), ~k=2,\cdots,K.			
			\label{process}
		\end{equation}
		where $h$ abbreviates \eqref{xeq}-\eqref{qeq}.
		For $k=1,\cdots,K$, the expected total cost of the process \eqref{process} with policy $\pi$ starting at $z$ is		
		\begin{equation}
			J(z,\pi) = \EEE \left[\sum_{k=1}^{K}{\sigma(u_k) + \rho(w_k)}  -\beta X_{K+1}  \right]
		\label{expectedCost}
		\end{equation}
		where the expectation is taken over the $\xi_k$ and the $Z_k$.	
	\end{definition}
	Observe that even if some discrete orders need to wait some time before being executed, their cost is incurred at the decision stage.	
	A policy $\pi^*$ is said to be optimal if,
	\begin{equation}
		J(z,\pi^*) = \min_{\pi \in \Pi(z)} J(z,\pi).
	\label{optimalExpectedCost}
	\end{equation}	
	i.e., it minimizes the expected cost \eqref{expectedCost}.
			
	\subsection{Value Functions and Stochastic Dynamic Programming Algorithm}
	To establish a dynamic programming principle it is necessary to extend the minimization problem \eqref{optimalExpectedCost} for general initial conditions. 
	
	The discrete random process departing from link $k_0$, given a initial condition $z \in \ZZZ$, and an admissible hybrid policy $\pi \in \Pi(z)$ is solution of
	\begin{eqnarray}
		\nonumber Z_{k+1} &=& h(Z_k,a_k,\xi_k), ~k=k_0,\cdots,K+1, \\
		Z_{k_0} &=& z.
	\label{processGeneralStart}
	\end{eqnarray}	
	Given instant cost functions $l_k$ for $k=k_0,\cdots,K$, the expected cost-to-go, function of the process \eqref{processGeneralStart} is 
	\begin{equation}
		C(z,k_0,\pi) = \EEE \left[\sum_{k=k_0}^{K}{\sigma(u_k) + \rho(w_k)}  -\beta X_{K+1}  \right].
	\label{expectedCostGeneralStart}
	\end{equation}		
	Define the value function of the optimal control problem of finding a policy that minimizes the expected cost \eqref{expectedCostGeneralStart} to be
	\begin{equation}
		v(z,k_0) = \min_{\pi \in \Pi(z)} C(z,k_0,\pi).
	\label{valueFunction}
	\end{equation}	
	A stochastic dynamic programming principle can be obtained from \eqref{valueFunction}.
	The next proposition allow the evaluation of the value function in a classical backward fashion.
	\begin{proposition}
		The value function \eqref{valueFunction} is calculated by the backward procedure:
		\begin{itemize}
			\item $v(z_{K+1},K+1) = -\beta x_{K+1};$
			\item For $k=K,\cdots,1$, 
			\begin{equation}
				v(z_k,k) = \min_{a \in \AAA(z_k)} \EEE [ \l(z_k,a_k,\xi_k) +  v(Z_{k+1},k+1) ].
			\label{SDPA}
			\end{equation}
		\end{itemize}
	\label{SDPAProposition}	
	\end{proposition}
	\begin{proof}
	The proof is derived from classical arguments. See for instance \cite{Bardi97optimalcontrol}.
	\end{proof}

	\subsection{Optimal Trajectory Reconstruction}
	
	This section details the procedure for synthesizing an optimal policy. 
	Evaluating the value function for all values of a discretized state space and all links is the first step into finding a sequence of control decisions that minimizes \eqref{expectedCost}. 
	Then, having particular initial conditions, one can synthesize a control sequence using an optimal trajectory reconstruction algorithm.	

	\begin{ALGORITHM} 
	
			\begin{enumerate}
			\item Step 1: Value function initialization.
				\begin{enumerate}
					\item Set $k=K+1$, set $\beta >0$;
					\item For all $z=(x,y,t,q)\in \ZZ$ set 
					$$v(z,K) = \left\{ \begin{array}{lll}
					\beta x_k & \text{if} & t = \delta, \\
					\infty & \text{if} & t < \delta.
					\end{array} \right.$$
				\end{enumerate}
			\item Step 2: Backward evaluation.
				\begin{enumerate}
					\item For $k=K,\cdots,1$, for all $z_k \in \ZZ$ evaluate:
					\begin{eqnarray*}
						v(z_k,k) &=& \min_{a \in \AAA(z_k)} \sigma(u) + \rho(w) + \\
					&&  \sum_{\xi \in \Omega(\eta_{k+1} )}v(z_{k+1},k+1) \PP(y_k,\xi) .
					\end{eqnarray*}
				\end{enumerate}
			\end{enumerate}
		\label{backwardEvaluation}
		\end{ALGORITHM} 
	Algorithm \ref{backwardEvaluation} is used for evaluation of \eqref{valueFunction} in all state space and at all links, where $\ZZ$ denotes a grid obtained from a discretization of $\ZZZ$.
	Once evaluating \eqref{valueFunction} in a grid of all state space and at all links, given a initial condition $Z_1 = z$, the optimal trajectory reconstruction is made as follows:  	
		\begin{ALGORITHM} 
			\begin{enumerate}
			\item Step 1: Initialization.
				\begin{enumerate}
					\item Set $k=1$, $z_1 = z$;
				\end{enumerate}
			\item Step 2: Control decision.
				\begin{enumerate}
					\item $a_k^* = \arg \min_{a \in \AAA(z) } \EEE[\sigma(u)+\rho(w) + v^\#(h(z_k,a,\xi),k+1)]$ ;
				\end{enumerate}				
			\item Step 3: Random realization and state evolution.
				\begin{enumerate}		
					\item $\xi = \xi^*$ ;
					\item $z_{k+1} = h(z_k,a_k^*,\xi^*)$ ;
				\end{enumerate}			 
			\item Step 4: Advance to next stage.
				\begin{enumerate}
					\item $k = k+1$ ;
					\item If $k=K+1$ terminate. Else, go to step 2 ;
				\end{enumerate}
			\end{enumerate}
		\label{trajectoryConstruction}
		\end{ALGORITHM}	
		Here,  $v^\#$ is the interpolate of the value function on the grid $\ZZ$ at $z_{k+1}$. 

	\section{Numerical Application}
	
	This section discusses the results obtained with simulated PMSs using a suitable speed profile. 
	Firstly, it states in what conditions the PMSs are synthesized and shows some of its main characteristics.
	Secondly, results of the performance of the calculated PMSs running in real speed cycles are shown.
	The performances of policies $\pi$ are discussed in the light of the economic gain relative to a pure EV strategy, i.e., a strategy that does not use the RE, namely
	\begin{equation}
		J^* = \frac{-\beta x_{K+1} + \sum_{k=1}^{K}{\sigma(u_k)+\rho(w_k)} }{- \beta x_{K+1}^{\text{EV}} }
		\label{criteria}
	\end{equation}•	
	where $ x_{K+1}^{\text{EV}}$ is the SOC when the RE is not activated.
	Also, the CPU time needed for policy evaluation is commented.

	In this section, three different ways of adding some lag/delay information in the synthesis of optimal controllers  are analyzed. 
	The first one is a rather classical strategy of penalization of the switch cost $\rho$ by a factor $\lambda>1$. 
	The second technique profits from the available NAV data in order to imbue the dynamic programming algorithm implicitly with time-related information -- recall that the dynamic programming algorithm is based on a route segmentation, being thus space-related. 
	Finally, the third method is the one given by algorithms \ref{backwardEvaluation} and \ref{trajectoryConstruction}, using the state variable $t$ and not relying on any particular structure of the problem itself, thus having the advantage of being flexible and general.
	
	\subsection{Deterministic PMSs and Simulations}
	
	In order to establish a benchmark this section analyzes controllers synthesized using NAV data and simulated in its ideal case, i.e. a deterministic case where the vehicle speed through the route follows exactly the NAV suggested speed - and hence, is perfectly known. 
	This result will then be compared to the same controller simulated in real driving conditions, using data record from the vehicle in order to assess the performance loss. 
	
	As pointed out, the deterministic approach consists in assuming that the driver will follow exactly the speed profile suggested by the NAV. 
	Hence, the variable $y$ can be removed from the hybrid state and the problem takes a deterministic form.
	For a fixed route, and thus a fixed NAV suggested speed profile, the value function is evaluated using the backward dynamic programming algorithm \ref{backwardEvaluation} where the expected value is dropped since it considers but one possible realization of the driver speed. 
	Then, setting initial conditions, the controller is synthesized using algorithm \ref{trajectoryConstruction}.
	Throughout all simulations, $\beta=2$.

\subsubsection{Penalization factor $\lambda$}
	Because the intermittent fast switching is undesirable, one possible approach that aims at decreasing the number of switches of the RE is the switch cost penalization. 
	This approach sets $\delta=0$ and includes a multiplicative penalty $\lambda>1$ at each switch of the RE, increasing the switch cost to $\lambda \rho(\cdot)$. 

	The CPU time is of $26$ seconds for all values of $\lambda$. 
	For $\lambda=1$ the system is in its nominal cost configuration without any lag/delay constraints. 
	Indeed, the strategy has a better performance \eqref{criteria} than when we add the penalty (cf table \ref{tab:PerformanceCriteriaForAlgorithms}), as one should expect. 
	For $\lambda=2$, a compromise is achieved between the criteria obtained and the number of switches executed. 
	The controller executes $8$ switches but, by inspecting the trajectories, one can remark that they are not sufficiently apart to respect the decision lag constraint. 
	For $\lambda>2$ the strategy does not change. It becomes optimal to turn on the RE right away and turn it off only near the end of the route because of the higher cost to restart it later. 
	The penalization technique in spite of reducing the number of switches effectively does not respect the control constraints needed in real applications. 
	Indeed, even if there are fewer switches, they are not necessarily spaced enough in time to respect the decision lag. 
	Additionally, in no case the execution delay constraint is respected. 
	This is an expected behavior as the penalization technique does not supply any direct information about any of the control constraints to the value function.
	 Despite the advantages -- simple numerical implementation and fast execution time -- this approach does not seem fit for the desired application. 
	\begin{table}[htbp]
	\caption{Performance criteria for penalized deterministic controllers with $\delta = 0$.}
	\label{tab:PerformanceCriteriaForAlgorithms}
	\centering
		\begin{tabular}{|l|c|c|}
		\hline
		$ $ & $J^*$  &switches\\ \hline
		 $\lambda=1$ & $1.1432 $ &22 \\ \hline
		 $\lambda=2$ & $1.0812 $ &8\\ \hline
		 $\lambda=20$ & $1.0534$ &2\\ \hline
		\end{tabular}
\end{table}

\subsubsection{Discretization grid adaptation}	
	
	When dealing with a deterministic formulation, the vehicle speed $y$ can be removed from the state vector. 	
	 In addition, as the vehicle speed is known beforehand, there is an unique relation between the position of any point through the route and the time it will be reached by the vehicle. 
	In this approach, that information is used for the discretization of the dynamic programming algorithm. 
	Conveniently, the discretization is made in intervals of time $\Delta_t$ that are sub-multiples of the delay and lag time $\delta$, satisfying the relation $\delta = m \Delta_t$, $m$ integer. 
	Indeed, proceeding as such, as one evaluates the value function at a node $k$, there is implicit information about the time since the vehicle departure, which is $k\Delta_t$.

	Observing the different performance levels for the values of $m$ considered, ranging from an improvement of $2.0\%$ to $5.6\%$, one can infer that the route discretization plays an important role in the performance level of the PMS.
	The major drawback of this approach is that, in spite of the constraints verification and the simplification obtained in algorithms \ref{backwardEvaluation} and \ref{trajectoryConstruction}, the CPU time becomes rapidly rather large (c.f. table \ref{tab:PerfGrid}).
	For that reason, such an approach is not judged fit for an application.
	Also, remark that even if the grid could be discretized in a more sophisticate fashion -- e.g., making $m$ vary to capture a  more volatile speed in some sectors -- such an approach would fail in the stochastic case, because the space-time transformation would not be available.
	Although this method cannot be applied in a stochastic scenario, it can still contribute to very relevant information concerning the route discretization.
	  
\begin{table}[htbp]
	\caption{Performance criteria and CPU time for deterministic controllers synthesized using adapted grid.}
	\label{tab:PerfGrid}
	\centering
		\begin{tabular}{|l|c|c|}
		\hline
		$$ & $J^*$ & CPU time (s)\\ \hline
		 $m=1$ & $1.0205$ & 5.28 \\  \hline
		 $m=2$ & $1.0563$ & 135.11 \\ \hline
		 $m=3$ & $1.0345$ & 402.72 \\ \hline
		 $m=4$ & $1.0515$ & 789.70\\ \hline	
		\end{tabular}
\end{table}

\subsubsection{General approach}		
	
	This subsection presents the results obtained for a deterministic approach that incorporates the lag and delay constraints on the control by setting $\delta = 120$ and taking into account the state vector variable $t$. 
	The relative performance level achieved in this case is $J^*=1.046$, which represents an improvement of $4.6\%$ over a purely electrical strategy.
	The CPU time for the controller synthesis is $204$ seconds, which is about the same order of the timewise grid discretization with $m=2$.
	However, when $m=2$, the performance level is $5.6\%$ better than an all EV policy and thus, better than the general approach.
	Again, this remark suggests that the route discretization plays a major role in the policy performance, as one might expect.
	
	\subsection{Stochastic PMSs and Simulations}
	The analysis of results of controllers synthesized for use with a real driving cycle is presented in this section.
	In this case, the vehicle future speed is known with certain probability.
	The value function is evaluated using algorithm \ref{backwardEvaluation} and a policy is synthesized using NAV speed data, both setting $\delta=0$ and using a penalization factor $\lambda$ and setting $\delta = 120s$ and including the time variable in the state vector.
	Next, this policy is simulated using several recorded real speed profiles.
	Each simulation is associated with a relative performance criteria $J^*$.
	The mean performance of the synthesized policies are grouped in table  \ref{tab:StoPerformanceCriteriaForAlgorithms}.
	Figure \ref{stoLambdaTkComparisonfig} also includes the standard deviation of the relative performances.
	When including a penalization factor $\lambda$, fewer switching is observed and the ICE is on the most throughout most of the path. 
	However, the activation and decision constraints are not properly taken care of by the controller, not being, thus, suitable for any realistic application.
	In average, the formulation including $t$ as a state variable has a better relative performance level achieving an improvement of $3.9 \%$ over a purely electric mode as well as respecting the control constraints. 
	\begin{table}[htbp]
	\caption{Performance criteria for penalized stochastic controllers with $\delta = 0$ and controller with $\delta = 120s$.}
	\label{tab:StoPerformanceCriteriaForAlgorithms}
	\centering
		\begin{tabular}{|l|c|c|c|}
		\hline
		$ $ & $\bar{J^*}$  &switchs (mean) & CPU time (s)  \\ \hline
		 $\lambda=1$ & $1.0339 $ &$23.50$ &	\multirow{3}{*}{2320}		\\ 
		 $\lambda=2$ & $1.0284 $ &$16.50$ &			\\ 
		 $\lambda=20$ & $0.9926$ &$10.06$ & 		\\ \hline
		 $\delta=120$ & $1.0386$  &$19.25$ & 4964 \\ \hline
		\end{tabular}
\end{table}

		\begin{figure}[tb]
		\begin{center}
			\includegraphics[width=.95\columnwidth]{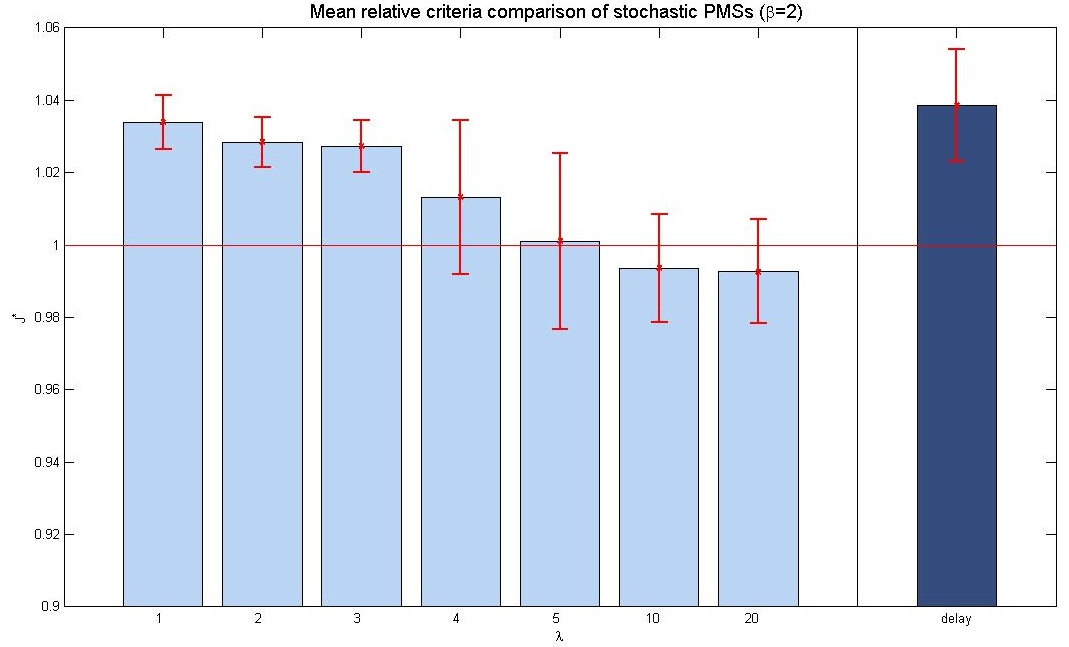} 
			\caption{Comparison between mean relative performance criteria for stochastic PMSs simulated using real data. }
			\label{stoLambdaTkComparisonfig}
		\end{center}
	\end{figure}

\section{Discussion}

	This work presents a stochastic dynamic programming algorithm for synthesizing optimal power management strategies, suitable for range-extender electric vehicles. 
	The stochastic model considers that along a geographic route the vehicle's speed is given by a mean value plus a random disturbance. 
	The hybrid dynamical system framework is used to state an hybrid optimal control problem. 
	The model considered presents two important features, namely, the utilization of information from the vehicle's navigation system the inclusion of physical constraints on the range-extender -- the activation delay and the decision lag. 
	Three methods for including the aforementioned control constraints are studied: penalization of the switching cost, discretization of the path in multiples of the delay/lag time and inclusion of a time state variable. 
	As a conclusion, the only method capable of taking into account the delay/lag constraints in a stochastic scenario is the general formulation, including a time state variable. 
	Results using real recorded data show that the inclusion of the time state variable allows an average improvement of $3.86\%$ in the performance level when compared to a purely electrical strategy.
	


\end{document}